\documentclass[12pt]{amsart}

\usepackage{amsmath,amssymb,amscd,amsthm,graphicx}

\numberwithin{equation}{section}
\theoremstyle{plain}
\newtheorem{theorem}[equation]{Theorem}

\newtheorem{proposition}[equation]{Proposition}
\newtheorem{lemma}[equation]{Lemma}
\newtheorem{corollary}[equation]{Corollary}

\theoremstyle{remark}

\theoremstyle{definition}
\newtheorem{definition}[equation]{Definition}

\newtheoremstyle{citing}
  {3pt}
  {3pt}
  {\itshape}
  {}
  {\bfseries}
  {.}
  {.5em}
  {\thmnote{#3}}

\theoremstyle{citing}
\newtheorem*{varthm}{}

\newcommand{\abs}[1]{\lvert#1\rvert}

\newcommand{\ra}{\rightarrow}
\newcommand{\restr}{\mbox{\Large \(|\)\normalsize}}

\newcommand{\A}{{\mathcal A}}

\newcommand{\C}{{\mathcal C}}

\newcommand{\F}{{\mathcal F}}

\newcommand{\N}{\mathbb N}

\newcommand{\Q}{\mathbb Q}
\newcommand{\R}{\mathbb R}

\newcommand{\U}{{\mathcal U}}

\newcommand{\ind}{\operatorname{Ind}}

\newcommand{\lip}{\operatorname{lip}}
\newcommand{\Lip}{\operatorname{Lip}}
\newcommand{\LIP}{\operatorname{LIP}}

\newcommand{\var}{\operatorname{var}}

\newcommand{\al}{\alpha}

\def\eps{\epsilon}
\def\ga{\gamma}

\def\la{\lambda}

\def\si{\sigma}

\def\defeq{:=}

\def\Xint#1{\mathchoice
   {\XXint\displaystyle\textstyle{#1}}%
   {\XXint\textstyle\scriptstyle{#1}}%
   {\XXint\scriptstyle\scriptscriptstyle{#1}}%
   {\XXint\scriptscriptstyle\scriptscriptstyle{#1}}%
   \!\int}
\def\XXint#1#2#3{{\setbox0=\hbox{$#1{#2#3}{\int}$}
     \vcenter{\hbox{$#2#3$}}\kern-.5\wd0}}

\def\av{\Xint-}
\def\dashint{\Xint-}

\begin{document}

\title[Primer on Differentiable structures]
{Differentiable structures on metric measure spaces: A Primer}
\author{Bruce Kleiner}
\address{Courant Institute of Mathematical Sciences\\251 Mercer Street\\New York, NY 10012,USA}
\email{bkleiner@cims.nyu.edu}
\thanks{Supported by NSF Grant DMS 0701515}
\author{John Mackay}
\address{Mathematical Institute, 24-29 St Giles'\\Oxford OX1 3LB, UK}
\email{john.mackay@maths.ox.ac.uk}
\date{\today}
\maketitle

\tableofcontents

\section{Introduction}

\subsection{Overview}
A key result of geometric function theory is 
Rademacher's theorem: any real-valued Lipschitz function on
$\R^n$ is differentiable almost everywhere.
In~\cite{Ch99}, Cheeger found a 
far-reaching
 generalization of this
result in the context of doubling metric measure spaces 
that satisfy a Poincar\'e inequality.  The goal of this primer
is to give a 
streamlined
 account of the construction of a 
measurable differentiable structure on such spaces, in the hopes of
providing an accessible introduction to this area of active research.
Our exposition is based on Cheeger's work, 
and incorporates a number of
clarifications due to Keith~\cite{KeDS}, as well as a few
of our own.  

The scope of this primer is limited to the foundational results
obtained in the first part of Cheeger's paper.
For a broader discussion of the historical and mathematical
context of this result, we refer the reader to the aforementioned
papers and to the survey of Heinonen~\cite{Hein}.

One of Cheeger's first achievements was to see that it is possible
to define a notion of differentiability in a metric space without any
additional algebraic structure.
A real valued function $f: \R^n \ra \R$ is differentiable at a point $x_0$ if
there is a some linear combination $L$ of 
the coordinate functions $x_i: \R^n \ra \R$, $i = 1, \ldots, n$, so
that the behavior of $f$ and $L$ near $x_0$ agree up to first order.
In other words, $f(\cdot)-f(x_0) = L(\cdot) - L(x_0) + o(d(\cdot,x_0))$.
Cheeger observed that this definition of differentiability 
with respect to 
a set of coordinate functions
 makes sense for real valued
functions on general metric measure spaces, 
where the 
role of the
coordinate functions is played by  suitable tuples of real valued Lipschitz functions.
Cheeger's version of Rademacher's theorem for metric measure spaces
asserts  that  there is a countable, full measure, disjoint collection of measurable subsets
equipped with  coordinate functions,
so that every Lipschitz function is differentiable almost everywhere
with respect to  the corresponding coordinate functions.

Of course, there is no reason for this conclusion to hold for every metric measure space.
Following Cheeger and Keith, we will show that it does hold when the space admits a $p$-Poincar\'e inequality.
%

Throughout this paper, $X = (X,d,\mu)$ denotes a metric measure space
with metric $d$ and measure $\mu$.  We will assume that
$\mu$ is Borel regular and doubling: there exists some constant $C$ so that 
for every $x \in X$ and $r>0$, $\mu(B(x,2r)) \leq C \mu(B(x,r))$.

\subsection{Statement of the theorem}

The discussion of differentiability given in the overview is formalized by the
following definition.

\begin{definition}[Cheeger, Keith]\label{def-diff-struct}
A {\em measurable differentiable structure} on a 
metric measure space $(X,d,\mu)$ is a countable collection of pairs 
$\{(X_\alpha,\mathbf{x}_\alpha)\}$, called {\em coordinate patches}, 
that satisfy the following conditions:
	
\begin{enumerate}
\item Each $X_\alpha$ is a measurable subset of $X$
with positive measure, and the union of the $X_\alpha$'s
has full $\mu$-measure in $X$.

\item Each $\mathbf{x}_\alpha$ is a $N(\alpha)$-tuple 
of Lipschitz functions
on $X$, for some $N(\alpha) \in \N$, where 
$N(\alpha)$ is bounded from above independently of $\alpha$.
The maximum of all the $N(\alpha)$ is called the {\em 
dimension} of the differentiable structure.
	
\item For each $\alpha$, $\mathbf{x}_\alpha = (x_\alpha^1, 
\ldots, x_\alpha^{N(\alpha)})$ \emph{spans the differentials 
almost everywhere for $X_\alpha$}, in the following sense:
For every  Lipschitz function $f:X \rightarrow \R$, there exists
a measurable function $df^\alpha: X_\alpha \rightarrow
 \R^{N(\alpha)}$ so that for $\mu$-a.e. $x \in X$,
\begin{equation}\label{eq-df}
	\limsup_{y \ra x} \frac{\left| f(y)-f(x)-df^\alpha(x)\cdot (\mathbf{x}_\alpha(y) 
	- \mathbf{x}_\alpha(x)) \right|}{d(x,y)} = 0.
\end{equation}
Moreover, $df^\alpha$ is unique up to sets of measure zero.
\end{enumerate}
\end{definition}

(Note that Keith uses the term ``strong measurable differentiable structure'' for the
above 
object.)


We can now state the main theorem,
which gives a sufficient condition for the existence of such a differentiable
structure.  (See \cite[Theorem 2.3.1]{KeDS} and \cite[Theorem 4.38]{Ch99}.)

\begin{theorem}\label{thm-main}
	If $(X,d,\mu)$ is a metric measure space that is doubling and supports
	a $p$-Poincar\'e inequality with constant $L\geq 1$  for some $p\geq 1$
	(see Definition \ref{def-pidefinition}),
	then $X$ admits a measurable differentiable structure with
	dimension bounded above by a constant depending only on $L$ and the doubling
	constant.
\end{theorem}

\subsection{Examples} 
We illustrate some of the possibilities for measurable
differentiable structures with the following examples.

	\begin{enumerate}
		\item \textbf{Euclidean spaces:} As a consequence of Rademacher's theorem,
			the metric measure space $\R^n$ (with the usual Euclidean metric
			and Lebesgue measure), has a measurable differentiable structure
			given by a single coordinate patch $(\R^n, \mathbf{x})$,
			where $\mathbf{x}$ is given by the coordinate functions 
			$\mathbf{x} = (x^1, \ldots x^n)$.
			
		\item \textbf{Carnot groups:} As a specific example of such a space,
			consider the Heisenberg group $\mathbb{H}$ of unipotent, 
			upper triangular, $3 \times 3$ matrices.
			As a set, $\mathbb{H}$ can be described by $\R^3 = \{ (x,y,z) \}$ with 
			a Carnot-Carath\`eodory metric and the usual Lebesgue measure.
			As a consequence of a theorem of Pansu~\cite{Pansu},
			this space carries a measurable differentiable structure
			with a single coordinate patch given by $\mathbf{x} = (x,y)$.
			In particular, 
			the dimension of the differentiable structure is two,
			the topological dimension of the space is three, and
			the Hausdorff dimension of the space is four,
			showing that all three may differ.
			
		\item \textbf{Glued spaces:} Consider the Heisenberg group $\mathbb{H} = \{(x,y,z)\}$ as above,
			and $\R^4 = \{(a,b,c,d)\}$ with its usual metric and measure.
			Note that these are both Ahlfors $4$-regular metric measure spaces.
			Choose an isometrically embedded copy of $\R^1$ in each --- 
			for example, the $x$-axis in $\mathbb{H}$, and the 
			$a$-axis
			 in $\R^4$ ---
			and let $X$ be the space formed by gluing $\mathbb{H}$ and $\R^4$ along these subsets.
			
			There is a natural geodesic path metric $d$ on $X$, and the measures combine to give
			an Ahlfors $4$-regular measure $\mu$ on $(X,d)$.
			By \cite[Example 6.19(a)]{HK98}, 
			$X$ admits a $p$-Poincar\'e inequality for $p > 3$.
			
			The space $(X,d,\mu)$ has a measurable differentiable structure
			with two coordinate patches,
			\[
				( X_1=\mathbb{H}, \;\mathbf{x_1} = (x, y)) \text{ and } 
					( X_2=\R^4, \;\mathbf{x_2} = (a,b,c,d) ).
			\]
			Notice that these coordinate patches
			are of different dimensions.
			
		\item \textbf{Laakso spaces:} For every $Q\geq 1$, Laakso builds an Ahlfors $Q$-regular space
			that admits a $1$-Poincar\'e inequality \cite{Laakso}.
			These fractal spaces have topological dimension one.
		
		\item \textbf{Bourdon-Pajot spaces:} These spaces arise as the boundary at infinity 
			of certain Fuchsian buildings that are 
			important examples
			 in geometric group theory.
			They are all homeomorphic to the Menger sponge, and admit a $1$-Poincar\'e inequality.
			
		\item \textbf{Limit spaces:}  The Gromov-Hausdorff limit of a sequence of Riemannian manifolds
			with Ricci curvature uniformly bounded from below, and diameter uniformly bounded
			from above, will admit a $1$-Poincar\'e inequality, even though it may no longer
			be a manifold.
	\end{enumerate}


\subsection{Organization of the paper}
In Section 2 we give an overview of the proof; readers with background in analysis on 
metric spaces may prefer to skip this, and refer back to it for definitions as needed.  
The proof of Theorem~\ref{thm-main} is given in Sections \ref{sec-prelim}-\ref{sec-pi-Liplip}.
In  Appendix \ref{sec-appendix}
we give a simpler proof of the well known result of Semmes~\cite[Appendix A]{Ch99}
that a Poincar\'e inequality on a complete doubling metric space implies that
the space is quasiconvex.  (That is, for all $x,y \in X$ there is a path
joining $x$ to $y$ of length at most $C d(x,y)$, for some uniform constant $C$.)

\begin{varthm}[Theorem~\ref{thm-quasiconvexity}]
	Suppose $X$ admits a $p$-Poincar\'e inequality (with constant $L\geq 1$) for some
	$p \geq 1$.  Then $X$ is $C$-quasiconvex, where $C$ depends only on $L,p$
	and the doubling constant.
\end{varthm}


\subsection{Acknowledgments} We thank Enrico Le Donne for comments on an earlier draft.

\section{Overview of the proof of Theorem~\ref{thm-main}}
Our purpose in this section is to give a nontechnical presentation of the proof
of Theorem~\ref{thm-main}, providing motivation, 
and a treatment more
accessible to readers from other areas.

\subsection{Finite dimensionality yields a measurable differentiable structure}
The first step in the proof of Theorem~\ref{thm-main} is a rather general argument showing that a $\si$-finite
metric measure space has a measurable differentiable structure provided it satisfies 
a certain finite dimensionality condition.   This involves two definitions:

\begin{definition}\label{def-dep}
An $N$-tuple of functions $\mathbf{f} = (f_1, \ldots, f_N)$, where
$f_i:X \rightarrow \R$ for $1 \leq i \leq N$, is {\em dependent 
(to first
order)} at $x \in X$ if there exists
$\lambda \in \R^n \setminus \{0\}$ so that
\begin{equation}
\label{eqn-nontrivialrelation}
\lambda \cdot \mathbf{f}(y) -\lambda \cdot \mathbf{f}(x) = o(d(x,y))
\end{equation}
as $y$ goes to $x$.
We denote the set where $\mathbf{f}$ is not dependent by $\ind(\mathbf{f})$.
\end{definition}

\begin{definition}\label{def-findim}
We say that in $(X,d,\mu)$
{\em the differentials have   dimension at most $N$} if 
 every $(N+1)$-tuple of Lipschitz 
functions is dependent almost everywhere.  We say that the  {\em differentials
have finite dimension} if they have dimension at most $N$ for some $N\in \N$.
\end{definition}

With these definitions, the first step of the proof is the following:

\begin{varthm}[Proposition \ref{prop-findim-coords}] \emph{} 
	If the differentials have  dimension at most $N_0$,
	then $X$ admits a measurable differentiable structure whose dimension is
	at most $N_0$.
\end{varthm}

\noindent
The proof of Proposition \ref{prop-findim-coords} is selection argument
analogous to the proof that a spanning subset
of a vector space contains a basis.  It works in considerable generality, e.g. for any $\si$-finite
metric measure space.

\subsection{Blow-up arguments, tangent spaces and tangent functions}
The remainder of the proof is devoted to showing that under the conditions of 
Theorem~\ref{thm-main}, the differentials have finite dimension.   To do this, one is faced
with analyzing the behavior of a tuple $(f_1,\ldots,f_N)$
of Lipschitz functions near a typical point in $X$, in order to produce nontrivial linear combinations
satisfying (\ref{eqn-nontrivialrelation}).
Following \cite{KeDS}, we approach this using a blow-up argument.   Blow-up arguments occur
in many places in geometry and analysis; the common features are a rescaling procedure which normalizes
some quantity of interest,  combined
with a compactness result which allows one pass to a limiting object which reflects the  asymptotic
behavior of the rescaled quantity.   Then one proceeds by studying the limiting object in order to derive 
a contradiction, or to establish a desired estimate.   
We point out that the blow-up argument is  not essential to this proof; it
is possible to work directly in the space itself.  However, in our view, the blow-up
argument clarifies and streamlines the proof.

For readers who are unfamiliar with this setting and/or blow-up arguments, 
we first illustrate the ideas using a single function.

To fix terminology and notation, we 
recall that a function $f:Y\ra Z$ between metric spaces is {\em $C$-Lipschitz} if 
$$
d_Z(f(p),f(q))\leq C\,d_Y(p,q)
$$
for all $p,q\in Y$,  while the {\em Lipschitz constant of $f$}
$$
\LIP(f)=\sup_{p,q\in Y,\,p\neq q}\frac{d_Z(f(p),f(q))}{d_Y(p,q)}
$$
is the infimal such $C$.  We let $\LIP(Y)$ denote the collection of real-valued Lipschitz functions $f:Y\ra \R$.

Now suppose $f\in \LIP(X)$ is a Lipschitz function, and $x\in X$.   To study the behavior of $f$ near $x$,
we may choose a sequence of scales $\{r_k\}$  tending to $ 0$, and consider the corresponding
sequence of rescalings of $(X,d)$, i.e. the sequence of metric spaces
$\{(X_k,d_k)\}$, where $X_k=X$ and $d_k=\frac{1}{r_k}d$.   
One then defines a sequence of functions $\{f_k:X_k\ra \R\}$ by rescaling $f$ accordingly:  $f_k=\frac{1}{r_k}f$.
Then $f_k$ has the same Lipschitz constant as $f$, and the behavior of $f$ in the ball $B(x,r_k)$
corresponds to the behavior of $f_k$ on the unit ball $B(x,1)\subset (X_k,d_k)$.  

Next, by passing to  a  subsequence, and using a suitable notion of convergence,
we may assume that the metric spaces $(X_k,d_k)$ converge to a 
(Gromov-Hausdorff) tangent space $(X_\infty,d_\infty)$,
and the functions $f_k:X_k\ra \R$ converge to a tangent function
$f_\infty:X_\infty\ra \R$ which is $\LIP(f)$-Lipschitz.
We will suppress the details for now, 
and refer the reader to Section \ref{sec-prelim} for
the notion of convergence (pointed Gromov-Hausdorff convergence)
and the relevant compactness theorems.   The space $X_\infty$ comes with a specified
basepoint $x_\infty\in X_\infty$, and the restriction of $f_\infty$ to the ball $B(x_\infty,R)$
is a limit of the restrictions $f_k\restr_{B(x_k,R)}$.

\subsection{Pointwise Lipschitz constants and tangent functions}
The tangent function $f_\infty$ is $\LIP(f)$-Lipschitz; however, 
since $f_\infty$ only reflects the behavior
of the original function $f$ near $x$, one is led to consider localized versions of the Lipschitz constant, 
as in the following definitions.

\begin{definition}\label{def-lip-constants}(Variation and pointwise Lipschitz constants)
\\
\noindent
Suppose $Y$ is a metric space, $x\in Y$, and $u\in \LIP(Y)$.
\begin{enumerate}
\item {\em The variation of $u$ on a ball $B(x,r)\subset Y$} is defined to be
	\begin{equation}
		\var_{x,r} u:= \sup \left\lbrace \frac{| u(y) - u(x) |}{r} \mid 
			y \in B(x,r) \right\rbrace.
	\end{equation}
\noindent
We always have $\var_{x,r} u \leq \LIP(u)$.
	
\item The {\em lower pointwise Lipschitz constant of $u$ at $x$} is	
	$$
		\lip_x u := \liminf_{r \rightarrow 0} \var_{x,r} u\,.
		$$
\item 	The {\em upper pointwise Lipschitz constant of $u$ at $x$} is		
$$
\Lip_x u := \limsup_{r \rightarrow 0} \var_{x,r} u\,.
$$
\end{enumerate}
\end{definition}

\bigskip
For any function $u:Y \ra \R$, and $x \in Y$,
we have $\lip_x u\leq \Lip_x u$.
In general, $\lip_x u$ and $\Lip_x u$ need not be comparable.
However, in the special case of $Y = \R^n$, 
if $x$ is a point of differentiability
of $u$, observe that $\lip_x u = \Lip_x u= |\nabla u(x)|$.

Returning to the tangent  function $f_\infty:X_\infty\ra \R$, one observes that the restriction
of $f_\infty$ to the ball $B(x_\infty,R)\subset X_\infty$ is the limit of the sequence $\{f_k\restr_{B(x_k,R)}\}$, 
which,  in turn, arises from rescaling $f\restr_{B(x,Rr_k)}$.   This leads to the bound 
\begin{equation}
\label{eqn-upperlowerlipschitzcontrolvar}
\lip_xf\leq\var_{x_\infty,R}f_\infty \leq \Lip_x f
\end{equation}
for all $R\in [0,\infty)$; in other words, the lower and upper
pointwise Lipschitz constants of $f$ at $x$ control the 
variation of 
the tangent function $f_\infty$ on balls centered at $x_\infty$.   

Using the fact that the measure on $X$
is doubling,  one can strengthen this assertion to:  For almost every $x\in X$,
every tangent function $f_\infty$ of $f$ at $x$ 
satisfies $\lip_xf\leq \var_{y,r}f_\infty\leq \Lip_xf$ for every $y\in X_\infty$, $r\in [0,\infty)$.  
The second inequality 
is equivalent to $\LIP(f_\infty)\leq \Lip_xf$.   However, for a general doubling metric measure space, the
quantity $\var_{x,r}f$ can fluctuate wildly as $r\ra 0$, which means that one could have
$\LIP(f_\infty)\ll \Lip_xf$.
A key observation of Keith  -- based on a closely related earlier observation of Cheeger -- is that when
$(X,d,\mu)$ satisfies a Poincare inequality, then this bad behavior can only occur when $x\in X$ belongs to 
a set of measure zero.

\begin{definition}[{\cite[(5)]{KeDS}}]\label{def-Lip-lip}
	We say $X$ is a \emph{$K$-Lip-lip space} if for every $f\in \LIP(X)$, 
	\begin{equation}\label{eq-Liplip}
			\Lip_x f \leq K\, \lip_x f 
	\end{equation}
for $\mu$-a.e. $x \in X$.
	If $X$ is a $K$-Lip-lip space for some $K > 0$, we say that
	\emph{$X$ is a Lip-lip space}.
\end{definition}

\begin{varthm}[Proposition~\ref{prop-pi-Liplip}] \cite[Prop.~4.3.1]{KeDS}\;\;
If $(X,\mu)$ is doubling, and satisfies a $p$-Poincar\'e
inequality, then $X$ is a $K$-Lip-lip space,
where $K$ depends only on the constants in the doubling
and Poincar\'e inequalities.
\end{varthm}
\noindent 
By Proposition~\ref{prop-pi-Liplip} it suffices to prove that the differentials have finite dimension
in any Lip-lip space.

\subsection{Tangent functions in Lip-lip spaces, and quasilinearity}
By (\ref{eqn-upperlowerlipschitzcontrolvar}), if  $(X,d,\mu)$ is a $K$-Lip-lip space,
and $f\in \LIP(X)$,  then for $\mu$-a.e. $x\in X$,  every tangent function
$f_\infty:X_\infty\ra \R$ of $f$ at $x$, and every $y\in X_\infty$, $r\in [0,\infty)$,
one has
\begin{equation}
\label{eqn-varlipLip}
\lip_xf\leq \var_{y,r}f_\infty\leq \LIP(f_\infty)\leq \Lip_xf\leq K\lip_xf\,,
\end{equation}
so in particular
\begin{equation}
\label{eqn-varLIP}
\var_{y,r}f_\infty\geq \frac{1}{K}\LIP(f_\infty)\,.
\end{equation}
Thus for any ball $B(y,r)\subset X_\infty$, the 
variation of $f_\infty$ on $B(y,r)$ agrees with the  global Lipschitz constant $\LIP(f_\infty)$
to within a factor of $K$.   This  leads to:

\bigskip
\begin{definition}
\label{def-quasilinear}
	A Lipschitz function $u:Z \ra \R$ on a metric space $Z$
	 is {\em $L$-quasilinear} if the variation
	of $u$ on every ball $B(x,r)$ satisfies 
	$$
	\var_{x,r} u \geq \frac{1}{L} \mathrm{LIP}(u)\,.
	$$
\end{definition}

In summary: when $X$ satisfies the $K$-Lip-lip condition, then for every $f\in \LIP(X)$
and $\mu$-a.e. $x\in X$, every 
tangent function of $f$ at $x$ is $K$-quasilinear.   

We need another version of the doubling condition appropriate to metric spaces:

\begin{definition}
\label{def-doublingmetricspace}
A metric space $Z$ is {\em $C$-doubling} if every ball can be covered by at most
$C$ balls of half the radius.  A metric space is {\em doubling} if it is $C$-doubling
for some $C$.
\end{definition}

The last key ingredient in the proof is:

\bigskip
\begin{varthm}[Lemma~\ref{thmquasilinearbound}]
For every $K,C$ there is an $N\in \N$ such that 
the space of $K$-quasilinear functions on a $C$-doubling metric  space $Z$ has dimension
at most $N$.
\end{varthm}

The Gromov-Hausdorff tangent spaces $X_\infty$ arising from a doubling metric measure space 
$X$ are all $C$-doubling for a fixed $C\in [1,\infty)$.  Therefore by 
Lemma~\ref{thmquasilinearbound} there  is  a uniform upper bound on the 
dimension of any space of $K$-quasilinear functions on any Gromov-Hausdorff tangent space of $X$.

A related finite dimensionality result appears in \cite{Ch99}.   We would like to point out
that a similar idea appears in the
earlier finite dimensionality theorem of Colding-Minicozzi \cite{coldingminicozzi}, also in the setting of 
spaces which satisfy a doubling condition and a Poincare inequality 
 (in  \cite{coldingminicozzi} the spaces are  Riemannian manifolds, though the smooth structure is not
used in an essential way).
In their paper,  the quasilinearity condition is 
replaced by a condition which compares the size of a function on a ball 
(measured in terms of normalized energy) with its size on subballs, and uses this
together with the Poincare inequality and doubling property to bound the dimension of 
a space of harmonic functions.

To complete the proof that the differentials have finite dimension in a $K$-Lip-lip space, we fix
an $n$-tuple of Lipschitz functions $\mathbf{f}=(f_1,\ldots,f_n)$ for some $n\in\N$.
Amplifying  the above reasoning,
there will be a full measure set of points $x\in X$ such that every set of tangent
functions $\mathbf{f}_\infty=
(f_{1,\infty},\ldots, f_{n,\infty})$ at $x$ spans a space of $K$-quasilinear
functions.   Thus when $n$ is larger than the dimension bound coming from 
Lemma~\ref{thmquasilinearbound}, there will be a nontrivial linear relation
$
\la\cdot \mathbf{f}_\infty =0\,
$
for some $\la\in \R^n\setminus \{0\}$.    This  implies 
that $f_1,\ldots,f_n$ are dependent at $x$.

\section{Preliminaries}\label{sec-prelim}

\subsection{Lipschitz constants}

Recall that we work inside a metric measure space $(X,d,\mu)$,
 where $\mu$ is a Borel regular measure on $X$.

We begin by making some observations about $\lip_x f$ and $\Lip_x f$ (see Definition~\ref{def-lip-constants}).
\begin{lemma}\label{lem-lip-Lip-meas}
	If $f:X \ra \R$ is Lipschitz, then $\lip_x f$ and $\Lip_x f$ are Borel measurable
	functions of $x$.
\end{lemma}
\begin{proof}
	For fixed $r>0$, we see that $\var_{x,r} f$ is a lower-semicontinuous function of $x$.
	(Note that $f$ is Lipschitz, so the variation over open balls cannot jump up as we approach
	a point.)

	We can rewrite $\Lip_x f$ as follows:
	\begin{equation}\label{eq-Lip-meas}
		\begin{split}
			\Lip_x f & = \lim_{r \ra 0} \sup\lbrace \var_{x,s} f \mid s<r \rbrace \\
				& = \lim_{r \ra 0} \sup \lbrace \var_{x,s} f \mid s<r, s \in \Q \rbrace .
		\end{split}
	\end{equation}

	The first equality holds by definition, and the second from the inequalities
	\[
		(s-\eps) \var_{x,(s-\eps)} f \leq s \var_{x,s} f 
			\leq (s+\eps) \var_{x,(s+\eps)} f.
	\]
	A countable supremum of measurable functions is measurable, and a pointwise limit of
	measurable functions is also measurable.  Therefore, by equation~\eqref{eq-Lip-meas}, we see that
	$\Lip_x f$ is a measurable function of $x$.  An analogous argument gives the same conclusion for
	$\lip_x f$.
\end{proof}

	In fact, for any $x \in X$, $\Lip_x(\cdot)$ defines a seminorm on $\LIP(X)$.
\begin{lemma}\label{lem-Lip-triangleineq}
	If $f:X \ra \R$ and $g:X \ra \R$ are Lipschitz, then for all $x \in X$ we have
	$\Lip_x (f+g) \leq \Lip_x f + \Lip_x g$.
\end{lemma}
\begin{proof}
	Fix $x \in X$.  Suppose we are given $\eps > 0$.  By equation \eqref{eq-Lip-meas}
	there exists $r>0$ so that for all $y \in B(x,r)$ we have
	\[ \frac{|f(y)-f(x)|}{d(x,y)} \leq \Lip_x f + \eps \quad \text{and} \quad
		\frac{|g(y)-g(x)|}{d(x,y)} \leq \Lip_x g + \eps.
	\]
 	We can find $y \in B(x,r)$ so that
 	\[ \Lip_x (f+g) \leq \frac{|(f+g)(y)-(f+g)(x)|}{d(x,y)} + \eps,\]
 	and applying the triangle inequality we see that
 	\[ \Lip_x (f+g) \leq (\Lip_x f + \eps) + (\Lip_x g + \eps) + \eps. \qedhere\]
\end{proof}

\begin{definition}
	Suppose $A \subset X$ is measurable.
	A point $x \in X$ is a {\em point of density of $A$} if
	\[
		\lim_{r \ra 0} \frac{\mu( B(x,r) \setminus A)}{\mu( B(x,r) )} = 0.
	\]
	
	A function $f:X \ra \R$ is {\em approximately continuous} at $x \in X$ if there exists a measurable
	set $A$, for which $x$ is a point of density, so that $f$ restricted to $A$ is continuous at $x$.
\end{definition}

\begin{lemma}[Theorem 2.9.13, \cite{Fed}]
	Assume $\mu$ is doubling.
	If $A \in X$ is measurable, then almost every point of $A$ is
	a point of density for $A$.
	
	If $f:X \ra \R$ is measurable, then $f$ is approximately continuous almost everywhere.
\end{lemma}
For the first part of this lemma, see also \cite[Theorem 1.8]{Hein-lect}.
The second part follows from Lusin's theorem.

\subsection{Gromov-Hausdorff convergence}

In this subsection we deal with metric spaces that do not a
priori come with a doubling measure; however, they are
doubling metric spaces, see Definition \ref{def-doublingmetricspace}.
  Every metric measure space
with a doubling measure is also a doubling metric space.
(For complete metric spaces the converse is also true,
but much less obvious.)


\begin{definition}
A sequence $\{(X_i,x_i)\}$ of pointed metric spaces 
{\em Gromov-Hausdorff converges} to a pointed metric
space $(X,x)$ if there is a sequence of maps
$\{\phi_i:X\ra X_i\}$, with $\phi_i(x)=x_i$ for all $i$,
such that for all $R \in [0,\infty)$ we have
$$
\limsup_{i\ra\infty} \big\{\; | d_{X_i}(\phi_i(y),\phi_i(z))-d_X(y,z) |
\;\mid\; y,z\in B(x,R)\subset X \big\} \;=\;0\,,
$$
and
$$
\forall \delta>0, \limsup_{i\ra\infty} \big\{\;d(y,\phi_i(B(x,R+\delta)))\;\mid\;
y\in B(x_i,R)\subset X_i \big\} \;=\;0.
$$

Such a sequence of maps is called a {\em Hausdorff approximation}.
\end{definition}

\begin{theorem}Every sequence of $C$-doubling pointed metric spaces $\{(X_i,x_i)\}$ has
a subsequence which Gromov-Hausdorff converges to a complete
$C$-doubling pointed metric space $(X,x)$.
\end{theorem}

	This follows from an Arzel\`a-Ascoli type of argument.  For each
	$\eps > 0$ and radius $r > 0$ 
	we can approximate $B(x_i,r) \subset X_i$ by a maximal $\eps$-separated
	net whose cardinality is independent of $i$.
	By repeatedly choosing subsequences we can
	ensure that these nets converge in the limit to a net of at most the same cardinality.
	To finish the proof, take further subsequences as $\eps \ra 0$ and $r \ra \infty$.
	For more details see \cite[Theorem 7.4.15]{BBI}.	

\begin{definition}
Let $\{(X_i,x_i)\}$ be a sequence of pointed metric spaces.
For a fixed 
countable
index set $\A$, suppose that $\{\F_i\}_{i\in\N}$ is
a sequence of collections of functions indexed by $\A$:
$$
\F_i=\{f_{i,\al}:X_i\ra \R\}_{\al\in\A}\,.
$$
Then the sequence of tuples 
$
\{(X_i,x_i,\F_i)\}_{i\in\N}
$
{\em Gromov-Hausdorff converges} to a tuple
$(X,x,\F)$, where $\F=\{f_\al:X\ra \R\}_{\al\in\A}$,
if there is a Hausdorff approximation
$\{\phi_i:X\ra X_i\}$ such that for all $x\in X$, $\al\in \A$,
$$
\lim_{i\ra\infty}\;f_{i,\al}(\phi_i(x))=f_\al(x).
$$
\end{definition}

If $\{(X_i,x_i)\}$ is a sequence of $C$-doubling
metric spaces, and $\{\F_i=\{f_{i,\al}:X_i\ra\R\}_{\al\in\A}\}$ 
is a sequence such that for every $\al\in\A$, both the Lipschitz
constants of the family $\{f_{i,\al}\}$ 
and the values $\{f_{i,\al}(x_i)\}$
are uniformly bounded,
then after passing to a subsequence if necessary,
the sequence of tuples $\{(X_i,x_i,\F_i)\}_{i\in\N}$ 
Gromov-Hausdorff converges.

\begin{definition}
	Suppose $X=(X,d)$ is a metric space, and $x \in X$.
\begin{enumerate}
\item 
	A pointed metric space $(X_\infty, d_\infty,x_\infty)$ is a
	{\em Gromov-Hausdorff (GH) tangent space} to $X$ at $x$ if it is the Gromov-Hausdorff
	limit of the pointed metric spaces
	$\{(X, d_i, x)\}_{i\in\N}$, where each $d_i = \frac{1}{r_i}d$ is
	the original metric $d$ rescaled by $r_i >0$, and
	the sequence $(r_i)$ converges to zero.
	
\item 	Suppose now that $\F=\{f_\al : X \ra\R\}_{\al\in\A}$ is a (countable)
	collection of functions on $X$.
	Then 
	\[ \U = \{u_{f_\al} : X_\infty\ra\R \}_{\al\in\A} \]
	is a collection of
	{\em tangent functions} of the functions $f_\al \in \F$ 
	at $x\in X$ if $\U$ is the Gromov-Hausdorff limit of the sequence of tuples
	$\{ (X,\frac{1}{r_i}d, x, \F_i) \}_{i\in\N}$, where
	\begin{gather*}
		\F_i = \left\lbrace f_{i,\al} : (X,d_i,x) \ra \R \right\rbrace_{\al\in\A}, \text{ and}\\
		f_{i,\al}(\cdot) = \frac{f_\al(\cdot)-f_\al(x)}{r_i}.
	\end{gather*}
	Since we used the same Hausdorff approximation and scaling factors
	for every $f_\al\in\F$,
	we say that the tangent functions are {\em compatible}.
\end{enumerate}
\end{definition}

We caution the reader that
the terminology used for GH tangent spaces varies: Cheeger calls them tangent cones,
and other objects tangent spaces, while Keith just calls them tangent spaces.

In general, the GH tangent spaces and functions one sees are highly dependent on the
sequence of scales chosen.

Since rescaling preserves doubling and Lipschitz constants, our previous discussion
has the following

\begin{corollary}\label{cor-tangents}

\begin{enumerate}
\item
	Doubling metric spaces have (doubling) GH tangent spaces at every point.
	
	\item Any countable collection $\F$ of uniformly Lipschitz functions on a doubling
	metric space $X$ has a compatible collection of tangent functions $\U$
	at every point of $X$.
	\end{enumerate}
\end{corollary}

\section{Finite dimensionality implies measurable 
differentiable structure}
	Our goal in this section is to prove
\begin{proposition}[cf.~Prop.~7.3.1, \cite{KeDS}]\label{prop-findim-coords}
	If the differentials have finite dimension of at most $N_0$ (see Definition
	\ref{def-findim}),
	then $X$ admits a measurable differentiable structure whose dimension is
	at most $N_0$.  
\end{proposition}

\begin{proof}
	We have $N_0$ fixed by the hypotheses.

	\begin{lemma}\label{lem-decomp}
		We assume the hypothesis of Proposition~\ref{prop-findim-coords}.
		Then, given any measurable $A \subset X$ with positive
		measure, we can find a measurable $V \subset A$ with positive measure
		and a function $\mathbf{x}:V\ra\R^N$
		so that $(V,\mathbf{x})$ is a coordinate patch.
	\end{lemma}

We now complete the proof, assuming Lemma \ref{lem-decomp}.
	Since $X$ is a doubling
	metric measure space it is $\sigma$-finite, so without loss of generality we may assume
	it has finite  measure.  Applying Lemma \ref{lem-decomp}, we 
	construct a sequence 
	of coordinate
	patches $(U_1,\mathbf{x}_1),\ldots,(U_i,\mathbf{x}_i),\ldots$ inductively as follows.
	Given $i\geq 0$ and charts $(U_1,\mathbf{x}_1),\ldots,(U_i,\mathbf{x}_i)$, if the union
	$\cup_{j\leq i}\,U_j$ has full
	measure in $X$, we stop; otherwise, let $\mathcal{C}$ be the collection
	of coordinate patches $(V,\mathbf{x})$ with 
	$V\subset X\setminus \cup_{j\leq i}\,U_j$, and choose  $(U_{i+1},\mathbf{x}_{i+1})
	\in \mathcal{C}$ such that 
	$\mu(U_{i+1})\geq \frac12\sup\{\mu(V)\mid (V,\mathbf{x})\in \C\}$.   If the resulting
	sequence of charts $\{U_j\}$ is infinite, then we have $\mu(U_j)\ra 0$ as $j\ra\infty$, because
	$\mu(X)<\infty$.
	The union $\cup_j\,U_j$ has full measure, else we 
	could choose a chart $(V,\mathbf{x})$
	where $V$ is a positive measure subset of $X\setminus \cup_j\,U_j$, and this 
	contradicts the choice of the $U_j$'s.

\end{proof}

	It remains to prove Lemma~\ref{lem-decomp}.  Before proceeding with this we
note that \eqref{eq-df} can be expressed more concisely as
\begin{equation}
	\Lip_x \left( f(\cdot) - df^\alpha(x) 
	\cdot \mathbf{x}_\alpha(\cdot) \right) = 0, 
	\text{ for } \mu\text{-a.e. } x 
	\in X_\alpha.
\end{equation}

	\begin{proof}[Proof of Lemma~\ref{lem-decomp}] 
		Consider the maximal $N$ so that there exists some positive
		measure set $V \subset A$, and some $N$-tuple of Lipschitz functions
		$\mathbf{x}$, so that $V \subset \ind(\mathbf{x})$,
		the set where $\mathbf{x}$ is not dependent.
		(Because of finite dimensionality, we have $0 \leq N \leq N_0$.)
		
		We want to show that $(V,\mathbf{x})$ is a coordinate patch.
		Take any Lipschitz function $f \in \mathrm{LIP}(X)$, and consider
		the $(N+1)$-tuple of functions $(\mathbf{x},f)$.  By the maximality of
		$N$ this is dependent almost everywhere in $V$,
		so for $\mu$-almost every $x \in V$ there exists 
		$\lambda(x) \in \R$ and $df(x) \in \R^N$
		so that
		\begin{equation}\label{eq-findcoords}
			\Lip_x \left( \lambda(x)f(\cdot)-df(x) \cdot \mathbf{x}(\cdot) \right) = 0.
		\end{equation}
		Since $V \subset \ind(\mathbf{x})$, we know that $\lambda(x) \neq 0$ almost everywhere,
		so, without loss of generality, we may assume that $\lambda(x) = 1$ everywhere.
		
		The uniqueness of $df$, up to sets of measure zero, follows from the fact that
		$\Lip_x(\cdot)$ is a semi-norm on the space of Lipschitz functions
		(Lemma \ref{lem-Lip-triangleineq}).
		Indeed, suppose that $df_1:V \ra \R^N$ and $df_2: V \ra \R^N$ both 
		satisfy \eqref{eq-findcoords} for almost every $x$.  Then
		\begin{multline*}
			\Lip_x\Bigl( \left(df_1(x)-df_2(x)\right) \cdot \mathbf{x}(\cdot) \Bigr)
			\\ \begin{aligned} &\leq \Lip_x\Bigl( f(\cdot)- df_1(x) \cdot \mathbf{x}(\cdot) \Bigr)
				+ \Lip_x\Bigl( f(\cdot)- df_2(x) \cdot \mathbf{x}(\cdot) \Bigr)
			\\ & = 0, \ \text{for } \mu \text{-a.e. } x.
			\end{aligned}
		\end{multline*}
		So, if $df_1$ and $df_2$ differed on a set of positive measure, then $\mathbf{x}$ would
		be dependent on that same set, but this is not possible.  Therefore $df_1 = df_2$ almost
		everywhere.
		
		It only remains to show that $df$ is measurable.  This follows if 
		$df^{-1}(K)$ is measurable for each compact $K \subset \R^N$.  We fix such a $K$ for
		the remainder of the proof.
		
		Consider the function $h_x : \R^n \ra \R$ given by
		\[ h_x(\la) :=  \Lip_x( f(\cdot)-\la\cdot\mathbf{x}(\cdot) ). \]
		The triangle inequality for $\Lip_x(\cdot)$ (Lemma~\ref{lem-Lip-triangleineq})
		implies that $h_x$ is continuous; in fact, for $\la,\la' \in \R^N$,
		\begin{align*}
			\abs{h_x(\lambda) - h_x(\lambda')} & \leq \Lip_x((\la-\la') \cdot \mathbf{x}) \\
				& \leq \sum_{1 \leq i \leq N} \abs{\la_i -\la_i'} \Lip_x(x_i) \\
				& \leq \Bigl(N \max_{1 \leq i \leq N} \LIP(x_i)\Bigr) \abs{\la-\la'}.
		\end{align*}

		Now set
 		\begin{equation*}
 			E := \left\lbrace x \in V \mid \exists \lambda \in K \text{ s.t. } h_x(\lambda) = 0 \right\rbrace.
 		\end{equation*}
		As we have seen, $df$ is uniquely defined up to a set of measure zero, so $df^{-1}(K)$
		equals $E$ less a set of measure zero.  Consequently, it suffices to show that $E$ is
		measurable.  Fix a dense countable subset $K'$ of $K$, and observe that
		\begin{align*}
			E & = \left\lbrace x \in V \mid \exists (\la_n)_{n\in\N} \subset K', \la\in K \text{ s.t. }
				h_x(\la_n) \ra 0, \la_n \ra \la \right\rbrace \\
				& = \bigcap_{n \in \N} \bigcup_{\lambda \in K'} 
				\left\lbrace x \in V \mid h_x(\la) < \tfrac{1}{n} \right\rbrace.
		\end{align*}
		The first equality follows from the continuity of $h_x$ and the density of $K'$
		in $K$.  The second equality follows from the compactness of $K$.
		Note that $h_x(\la)$ is a measurable function of $x$ for 
		fixed $\la \in \R^N$ (applying Lemma~\ref{lem-lip-Lip-meas}).  Therefore,
		$E$ is a measurable set, and we are done.
	\end{proof}

We note one consequence of the above proof.
\begin{lemma}\label{lem-sf-meas}
	Suppose $(X,d,\mu)$ is a Borel regular metric measure space, and that
	$\mathbf{x}$ is an $N$-tuple of real-valued Lipschitz functions on $X$.
	Then $\ind(\mathbf{x})$, the set where $\mathbf{x}$ is not dependent to
	first order, is a measurable set.
\end{lemma}
\begin{proof}
	This follows from the same argument that we used to prove that $E$ was measurable in
	the previous lemma.  Notice that
	\begin{align*}
		X \setminus \ind(\mathbf{x}) & = 
			\big\{ x \in X \mid \exists \lambda \in \R^N \setminus \{0\} 
				\text{ s.t. } \Lip_x(\lambda \cdot \mathbf{x}) = 0 \big\} \\
			& = \bigcup_{n \in \N} E_n,
	\end{align*}
	where
	\[
		E_n = \big\{ x \in X \mid \exists \lambda \in \R^N, 
				\text{ s.t. } \tfrac{1}{n} \leq |\lambda| \leq n, \text{ and }
				\Lip_x(\lambda \cdot \mathbf{x}) = 0 \big\}.
	\]
	Since the annulus $\{ \lambda \in \R^N \mid \tfrac{1}{n} \leq |\lambda| \leq n\}$
	is compact, the argument at the end of the proof of Lemma~\ref{lem-decomp}
	shows that $E_n$ is measurable, and this completes the proof.
\end{proof}

\section{A $\Lip$-$\lip$ inequality implies finite dimensionality}

In this section we prove the following statement, which perhaps is the
heart of the theorem.  
Throughout this section,
$(X,d,\mu)$ is a doubling metric measure space with
a $K$-Lip-lip bound, for fixed $K > 0$.
\begin{proposition}[Prop.~7.2.2, \cite{KeDS}]\label{prop-Liplip-findim}
	There exists an $N_0$, depending only on $K$ and the doubling constant,
	so that any $(N_0+1)$-tuple $\mathbf{f}$ of Lipschitz functions
	is dependent almost everywhere.
	
	In other words, $(X,d,\mu)$ is finite dimensional.
\end{proposition}

Suppose we fix $N$ Lipschitz functions $\mathbf{f} = (f_1,\ldots, f_N)$.
By Lemma~\ref{lem-sf-meas}, we know that $\ind(\mathbf{f})$,
the set of points where $\mathbf{f}$ is not dependent, is measurable,
and we assume that it has positive measure.  The proposition
will be proved if we can find a bound $N\leq N_0$.

Let $\F$ be the countable collection of all rational linear
combinations 
\[
	\F = \{ \lambda \cdot \mathbf{f} \mid \lambda \in \Q^N \} \subset \LIP(X).
\]
The rough idea is that we can take tangents to $X$ and $\F$ at a suitable point
to get a vector space of uniformly \emph{quasilinear} functions
that is, Lipschitz functions whose variation on any ball is comparable to
their Lipschitz constant.  The doubling condition then provides an
an upper bound for the size of this vector space, and hence of $N$.

\subsection{Finding good tangent functions}

\begin{definition}
If $f$ is  a Lipschitz function  and $\eps>0$,
 a subset $Y\subset X$
is {\em $\eps$-good for $f$} if there is an $r_0\in (0,\infty)$ such that 
if $r\in (0,r_0)$ and $x\in Y$, then
\begin{equation}\label{eq-eps-good}
\frac{1}{K}\Lip_xf-\eps\leq \lip_x f-\eps
\leq \var_{x,r}\,f\leq \Lip_x f+\eps\,.
\end{equation}
The set $Y$ is {\em good for $f$} if it is $\eps$-good
for $f$, for all $\eps>0$.
If $\F$ is a collection of functions, then
the set $Y$ is {\em $\eps$-good for $\F$} (respectively
good for $\F$) if
it is $\eps$-good (respectively good) for every $f\in \F$.
\end{definition}

\begin{lemma}\label{lem-good-sets}
	Suppose $Y_0\subset X$ is a measurable subset of finite measure and $\eps>0$.
	Given a Lipschitz function $f$, for all $\delta >0$ there exists $Y \subset Y_0$ 
	so that $\mu(Y_0 \setminus Y) < \delta$ and $Y$ is $\eps$-good for $f$.
	
	Consequently, given a countable collection of Lipschitz functions $\mathcal{F}$,
	neglecting a set of arbitrarily small measure
	we can find $Y \subset Y_0$ so that $Y$ is good for $\mathcal{F}$.
\end{lemma}

\begin{proof}[Proof of Lemma~\ref{lem-good-sets}]
	The first inequality of \eqref{eq-eps-good} follows, almost everywhere,
	from the Lip-lip inequality \eqref{eq-Liplip}.

	We saw $\Lip_x f$ was a measurable function of $x$ using the pointwise
	convergence of functions in equation \eqref{eq-Lip-meas}.  (A similar equation
	holds for $\lip_x f$.)
	By Egoroff's theorem, after neglecting a subset of arbitrarily 
	small measure, we may obtain a measurable set $Y\subset Y_0$ where
	the convergence is uniform. 
	 This completes the proof of
	\eqref{eq-eps-good}.
\end{proof}

As in the introduction to this section, we fix $N$ Lipschitz functions 
$f_1,\ldots, f_N$, and let $\F$ be the countable collection of all 
rational linear combinations of these functions.

Let $Y_0\subset X$
be a finite measure subset.  By the above
reasoning, and Lusin's theorem,  after neglecting a
subset of arbitrarily small measure, we may obtain a
measurable subset $Y_1 \subset Y_0$ such that 
\begin{itemize}
\item
for all $f\in\F$, the function $\Lip_x f$ (viewed as a function of $x$) is continuous on $Y_1$, and
\item the set $Y_1$ is  good for $\F$. 
\end{itemize}

\begin{lemma}\label{lem-blow-up-to-ql}
Suppose $x\in Y$ is a density point of the above set $Y_1$. 
Let $X_\infty$ denote a tangent of $X$ at $x$,  and $\{u_f:X_\infty\ra \R\mid f\in\F\}$ 
denote a compatible collection of tangent functions.
Then 
\begin{enumerate}
\item
\quad$\LIP u_f\leq \Lip_x f$.

\item 
\quad For every $p\in X_\infty$, and every $r\in(0,\infty)$,
$$
\Lip_x f\leq K\var_{p,r}\,u_f\,.
$$
\end{enumerate}
Thus the functions $u_f$ are uniformly quasilinear (Definition \ref{def-quasilinear}),
and have global Lipschitz constant comparable to $\Lip_xf$.
\end{lemma}

\begin{proof}
	Fix a Hausdorff approximation
	\[ \{ \phi_i: (X_\infty,d_\infty,x_\infty) \ra (X,d_i,x)\}_{i\in\N}, \]
	where $d_i = \frac{1}{r_i}d$ and $r_i \ra 0$.
	As $x$ is a point of density for $Y$, and $\mu$ is doubling, we can
	find maps
	\[ \{ \phi_i': (X_\infty,d_\infty,x_\infty) \ra (Y,d_i,x)\}_{i\in\N},\]
	so that $d_i(\phi_i(\cdot),\phi_i'(\cdot))$
	converges to zero uniformly on compact sets.
	
	Suppose we fix $p \neq q$ in $X_\infty$, $f\in\F$, and $\eps>0$.
	Let $p_i = \phi_i'(p), q_i = \phi_i'(q) \in Y$. 
	Notice that $d(p_i,q_i)\ra 0$ as $i \ra \infty$.
	
	For all sufficiently large $i$ we have, using the fact that $f$ is Lipschitz,
	\begin{equation}\label{eq-goodsets}
		\frac{\abs{u_f(p)-u_f(q)}}{d_\infty(p,q)}
		\leq \frac{\abs{\tfrac{1}{r_i}f(p_i) - \tfrac{1}{r_i}f(q_i)}}{
				\tfrac{1}{r_i}d(p_i,q_i)} +\eps.
	\end{equation}
	
	Since $Y$ is $\eps$-good for $f$, there exists $r_0$ so that
	\eqref{eq-eps-good} holds.  To prove (1), use \eqref{eq-goodsets}
	to see that
	\begin{align*}
		\frac{\abs{u_f(p)-u_f(q)}}{d_\infty(p,q)}
			& \leq \var_{p_i,(1+\eps)d(p_i,q_i)} f + \eps \\
			& \leq \Lip_{p_i} f + 2\eps \text{, by \eqref{eq-eps-good}.}
	\end{align*}
	Since $\Lip_x f$ is continuous on $Y$, and $p_i \ra x$ in the metric $d$,
	we see that
	\[ \frac{\abs{u_f(p)-u_f(q)}}{d_\infty(p,q)} \leq \Lip_x f + 2\eps, \]
	but $\eps$ was arbitrary, and so were $p$ and $q$, so (1) is proved.
	
	To see (2), fix $\eps>0$ and take $p_i$ as before.
	Now choose $a_i \in B(p_i,(r-\eps)r_i) \subset (X,d)$ so that
	\[
		\var_{p_i,(r-\eps)r_i} f \leq \frac{\abs{f(p_i)-f(a_i)}}{(r-\eps)r_i}+\eps.
	\]
	For sufficiently large $i$, at a cost of adding another
	$\eps$ to the right	hand side, we can assume that
	$a_i \in Y$, and that $a_i = \phi_i'(v_i)$, for some
	$v_i \in B(p,r)$.
	Furthermore, since $f \circ \phi_i' :X_\infty\ra\R$
	converges to $u_f$ pointwise, and these functions are uniformly
	Lipschitz, the convergence is uniform on compact sets.  Therefore
	for sufficiently large $i$,
	\begin{equation}\label{eq-var-approx} 
		\var_{p_i,(r-\eps)r_i} f \leq 
			\frac{\abs{u_f(p)-u_f(v_i)}}{r-\eps}+3\eps
		\leq \frac{r}{r-\eps} \var_{p,r} u_f +3\eps. 
	\end{equation}
	But by the continuity of $\Lip_x f$ on Y and equation \eqref{eq-eps-good},
	\begin{equation}\label{eq-Lip-var-bound}
		\Lip_x f = \lim_{i\ra\infty} \Lip_{p_i} f
			\leq \lim_{i\ra\infty} K \left( 
				 \var_{p_i,(r-\eps)r_i} f + \eps \right).
	\end{equation}
	Since $\eps>0$ was arbitrary, after combining
	\eqref{eq-var-approx} and \eqref{eq-Lip-var-bound},
	we are done.
\end{proof}

\subsection{Bounding the dimension of the space of tangent functions}

We say that $T \subset X$ is a \emph{$c$-net} if
the $c$-neighborhood of $T$ is $X$.  If in addition every 
two distinct points of $T$ are at least $c$ apart,
we say that $T$ is a \emph{(maximal) $c$-separated net}.

\begin{lemma}
\label{thmquasilinearbound}
Suppose $V$ is a linear space of $K$-quasilinear
functions on a metric space $Z$.
\begin{enumerate}
\item 
\label{itemgeneral}If  some
$r$-ball in $Z$ contains a finite  $\frac{r}{4K}$-net $T$,
then $\dim V\leq |T|$. 
\item If $Z$ is $C$-doubling, 
then $\dim V \leq (AK)^{\log_2 C}$,
where $A$ is a universal constant. 
\end{enumerate}
\end{lemma}

{\em Proof of (1).}
After rescaling, we may assume that $r=1$.  Let $B=B(x,r)=B(x,1)$,
and let $T\subset B$ be a maximal
$\frac{1}{4K}$-separated net.

Suppose $u\in V$ is in the kernel of the restriction map
$V\subset L^\infty(B)\ra L^\infty(T)$.  If $x\in B$, there is a $t\in T$
with $d(t,x)< \frac{1}{4K}$, so
\begin{align*}
	|u(x)| & =|u(x)-u(t)|\leq \LIP(u)\,d(x,t) \\
		& \leq K(\var_{B}\,u) \cdot\frac{1}{4K}
			\leq \frac12 \left\| u\restr_B \right\|_{L^\infty}\,.
\end{align*}
This implies that 
$$
\left\| u\restr_B \right\|_{L^\infty}\leq \frac12 \left\| u\restr_B \right\|_{L^\infty},
$$
forcing $\|u\restr_B\|_{L^\infty}=0$.  By quasilinearity,
we get $u\equiv 0$.   Thus the restriction map is injective,
and $\dim V\leq \dim L^\infty(T)=|T|$.  

{\em Proof of (2).}
The $C$-doubling condition implies that if $B\subset Z$
is a unit ball, there is a $\frac{1}{4K}$-net $T\subset B$
with $|T| \leq (16K)^{\log_2 C}$.
Then Part (1) applies. \qed

\subsection{Bounding the dimension of the differentials}
As stated in the introduction to this section,
we assume that $\ind(\mathbf{f})$ is a measurable set of
positive measure.

Using Lemmas~\ref{lem-good-sets} and \ref{lem-blow-up-to-ql}
(applied to $Y_0 = \ind(\mathbf{f})$)
we can take a GH tangents to $X$ and $\F$ at some $x \in \ind(\mathbf{f})$
to find a GH tangent space $Z = X_\infty$ with
a compatible family of tangent functions $\{u_f\mid f\in \F\}$.
Note that this family 
is the span over $\Q$ of $\{u_{f_1},\ldots,u_{f_N}\}$.
Since these are all $K$-quasilinear for a fixed $K$,
the same is true of the span over $\R$ of
$\{u_{f_1},\ldots,u_{f_N}\}$.  

We suppose for a contradiction that
$N > (AK)^{\log_2 C}$.
By Lemma \ref{thmquasilinearbound}, the functions 
$\{u_{f_1},\ldots,u_{f_N}\}$ satisfy a nontrivial linear
relation $\sum_i\;b_i u_{f_i}=0$ with real coefficients.
Approximating the vector $b=(b_1,\ldots,b_N)\in \R^N$
with a sequence of rational vectors 
$(a_{1,k},\ldots,a_{N,k})\in \Q^n$,
we get that the sequence of linear combinations
$v_k\defeq \{\sum_i\;a_{i,k}\;u_{f_i}\}$ tends to 
zero uniformly on bounded subsets of $X_\infty$.  
From the construction of the $u_f$'s,
this means that $\Lip_x( \sum_i\;a_{i,k}\; f_i)\ra 0$.  But
then
\begin{multline*}
\Lip_x\biggl(\sum_i\;b_i\,f_i\biggr)
\\ \leq \limsup_{k\ra\infty}\;
\left(\Lip_x\biggl(\sum_i\;a_{i,k}\,f_i\biggr)
+\Lip_x\biggl(\sum_i(b_i-a_{i,k})\,f_i\biggr)\right)
\\
\leq \limsup_{k\ra\infty}\;
\left(\Lip_x\biggl(\sum_i\;a_{i,k}\,f_i\biggr)
+\sum_i\;|b_i-a_{i,k}|\,\LIP f_i\right) \;=\;0\,.
\end{multline*}
  Hence the $f_i$'s
are dependent to first order at $x$, contradicting our assumption.
\qed

\section{A Poincar\'e inequality implies a $\Lip$-$\lip$ inequality}\label{sec-pi-Liplip}

\begin{definition}
\label{def-pidefinition}
	Fix $p \geq 1$. 
	A metric measure space $(X,\mu)$ admits a $p$-Poincar\'e inequality
	(with constant $L \geq 1$)
	if every ball in $X$ has positive and finite measure, and for every
	$f \in \LIP(X)$ and every ball $B = B(x,r)$
	\begin{equation}\label{eq-pi}
		\dashint_B | f - f_B | d\mu \leq 
			Lr \left( \dashint_{LB} (\lip_x f)^p d\mu(x) \right)^{1/p}.
	\end{equation}
\end{definition}

This is is equivalent to the usual definition
of a Poincar\'e inequality (see \cite[(4.3)]{Ch99}, and \cite{KeMeasPI}).
(Note that $\lip_x f$ is an upper gradient for $f$.)

The goal of this section is the following
proposition.

\begin{proposition}[Prop.~4.3.1, \cite{KeDS}]\label{prop-pi-Liplip}
	Suppose $X$ admits a $p$-Poincar\'e inequality (with constant $L\geq 1$) for some
	$p \geq 1$.  (See Section~\ref{sec-pi-Liplip} for the definition.)
	Then $X$ has a $K$-Lip-lip bound \eqref{eq-Liplip}, where $K$
	depends only on $L$ and the doubling constant of $X$.
\end{proposition}

We will use the following: 
\begin{lemma}
The space $(X,\mu)$ is given as above.
Suppose $A<\infty$ and $\eps>0$ are fixed constants.
If $u:X\ra \R$ is a Lipschitz function, and $x\in X$ is an 
approximate continuity point for 
$\lip u:X\ra \R$, then 
there exists $r_0=r_0(u,x,A,\eps)>0$ such that if
$r \leq r_0$, $y,y'\in B(x,Ar) \subset X$ and $d(y,y')\leq r$, then
\begin{equation}
	\left |\av_B u-\av_{B'}u\right|\leq C_1r\,(\lip_x u + \eps),
\end{equation}
where $B\defeq B(y,r),\,B'\defeq B(y',r)$,
and where $C_1=C_1(X,\mu)<\infty$ is a suitable constant.
\end{lemma}

\begin{proof}
Set $\hat B\defeq B(y,2r)$, so $B,B'\subset \hat B$.
Then  we have 
\begin{equation}
\label{doublingpi}
C_2\left|\av_B u-\av_{B'}u\right| \leq \av_{\hat B}|u-u_{\hat B}|
\leq 2Lr\left(\av_{L\hat B} (\lip u)^p\right)^\frac{1}{p},
\end{equation}
where $C_2>0$ depends only on the doubling constant for $\mu$,
and the second inequality comes from the Poincar\'e inequality for 
$(X,\mu)$.
Since $\lip u \leq \LIP(u)$ everywhere,
and $x$ is an approximate continuity point of $\lip u$, when
$r$ is sufficiently small we have 
\begin{equation}
\label{2lipu}
\left(\av_{L\hat B}(\lip u)^p\right)^{\frac{1}{p}}\leq \lip_x u +\eps.
\end{equation}
Combining (\ref{doublingpi}) and (\ref{2lipu}) gives the lemma.
\end{proof}

\begin{proof}[Proof of Proposition~\ref{prop-pi-Liplip}]
Since $\mathrm{lip}f$ is Borel it is approximately continuous almost
everywhere.  Let $x\in X$ be an approximate continuity point for
$\lip f$, and fix $\la \in (0,1)$, $\eps \in (0,1)$.

Since $(X,\mu)$ is doubling, its completion $\bar X$ equipped with the measure
$\bar \mu$ defined  by $\bar \mu(Y)=\mu(Y\cap X)$ is also doubling.  If $\bar x\in \bar X$, 
$r\in (0,\infty)$, we may find $x\in B(\bar x,r)\cap X$.
Then using the doubling property of $\bar \mu$ and the $p$-Poincare
inequality for $X$, we get
$$
\av_{B(\bar x,r)}|f-f_{B(\bar x,r)}|d\bar \mu \leq C\av_{B(x,2r)}|f-f_{B(x,2r)}|d\bar \mu
$$
$$
\leq 2CLr\left(\av_{B(x,2Lr)}(\lip_xf)^pd\mu(x)\right)^{\frac1p}
\leq 2CLr\left(\av_{B(\bar x,(2L+1)r)}(\lip_xf)^pd\bar\mu(x)\right)^{\frac1p}
$$
where $C$ depends only on the doubling constant of $\mu$.
Hence $(\bar X,\bar\mu)$ also satisfies a $p$-Poincare inequality, and is
quasiconvex by Theorem~\ref{thm-quasiconvexity}.  Therefore,
given $r > 0$ and $y\in B(x,r)$, by the quasiconvexity
of $\bar X$,
 there is a chain of points $x=p_1,\ldots,p_k=y$ in $X$,
where $d(p_i,p_{i+1})\leq \la r$ and $k\leq \frac{Q}{\la}$,
for some $Q$ that depends only on $X$.  Set $B_i\defeq B(p_i,\la r)$.
Then 
\begin{multline}\label{ballchain}
	|f(y)-f(x)| \leq \\ \left| f(x)-\av_{B_1}f \right|+
		\sum_{1\leq i<k}\left|\av_{B_{i+1}}f
-\av_{B_i}f\right| +
		\left| \left(\av_{B_k}f\right)-f(y)\right|.
\end{multline}
By the lemma, when $r$ is sufficiently small we have
$$
\left| \av_{B_{i+1}}f-\av_{B_i}f\right|\leq C_1\la r (\lip_x f + \eps),
$$
so
\begin{align*}
	|f(y)-f(x)| & \leq \left(\frac{Q}{\la}\right)(C_1\la r (\lip_x f+\eps))
		+2\la r \LIP(f) \\
	& = \left(QC_1 (\lip_x f + \eps)+2\la \LIP(f)\right)r.
\end{align*}
Thus $\Lip_x f \leq QC_1 \lip_xf+QC_1\eps+2\la \LIP(f)$ and, since
$\la,\eps>0$ were arbitrary, this proves the proposition.
\end{proof}

\appendix

\section{A Poincar\'e inequality implies quasiconvexity}
\label{sec-appendix}

As mentioned in the introduction, in this appendix we give a 
simpler proof of the following theorem of Semmes
\cite[Appendix A]{Ch99}.
A similar argument can be found in~\cite[Section 6]{KePI}.

\begin{theorem}
\label{thm-quasiconvexity}
Let $(X,d,\mu)$ be a complete, doubling metric measure space
satisfying a Poincar\'e inequality.  Then $X$ is $\la$-quasiconvex,
where $\la$ depends only on the data (doubling constant and
constants in the PI).
\end{theorem}

The main step in the proof of Theorem \ref{thm-quasiconvexity}
is:
\begin{lemma}
\label{lem-halfgap}
There is a constant $C\in (0,\infty)$ such that if
$p,q\in X$, and $r=d(p,q)$, then there is a path
of length at most $C\,r$ from $B(p,\frac{r}{4})$ to 
$B(q,\frac{r}{4})$.
\end{lemma} 

Assuming the lemma, the proof goes as follows.  Pick 
$x,x'\in X$, and apply the lemma to obtain a path $\ga$
of length at most $Cd(x,x')$, 
such that ``total gap'' $d(x,\ga)+d(\ga,x')$ is at
most $\frac12 d(x,x')$.   Now apply the lemma to each
of the gaps, to get two new paths, and so on.
The total gap at each step is at most
half the total gap at the previous step, 
and the total additional path produced  is at most $C$-times
the gap left after the previous step.  The closure of the 
union of the resulting collection of paths contains a path
from $p$ to $q$ of length at most $2\,C\,d(x,x')$.

\bigskip
Before proving Lemma \ref{lem-halfgap}, we make the following 
definition:

\bigskip
\begin{definition}
\label{def-eps-path}
An {\em $\eps$-path} in a metric space $X$ is a 
sequence of points $x_0,\ldots,x_k\in X$ such that
$d(x_{i-1},x_i)<\eps$ for all $i\in\{1,\ldots,k\}$;
the {\em length} of the $\eps$-path is $\sum_{i}\;d(x_{i-1},x_i)$.
\end{definition}

To prove Lemma \ref{lem-halfgap}, we will show that for all  $\eps\in (0,\infty)$,
 there is an $\eps$-path from $B(p,\frac{r}{4})$ to 
$B(q,\frac{r}{4})$
of length at most $C\;d(p,q)$;
then a  variant of the 
Arzel\`a-Ascoli theorem applied to a sequence of discrete
paths implies
that there is a path of length at most $C\,d(p,q)$
from $B(p,\frac{r}{4})$ to $B(q,\frac{r}{4})$. 

Fix $\eps\in (0,\infty)$, and  
define  $u:X\ra [0,\infty]$ by setting
$u(x)$ equal to the infimal length of an $\eps$-path from 
$B(p,\frac{r}{4})$ to $x$.  For $A\in (0,\infty)$, let
$u_A\defeq \min(u,A)$.  Then $u_A$ is is a continuous function which
is locally $1$-Lipschitz; in particular the constant function
$\rho\equiv 1$ is an upper gradient for $u_A$.  The Poincar\'e
inequality applied to $u_A$ and $B(p,\frac{5r}{4})$ implies
that $u_A$ is $\leq C\,r$ somewhere in $B(q,\frac{r}{4})$,
where $C$ depends only on the data of $X$.
Since this is true for $A>Cr$, the desired $\eps$-path exists.
\qed

\bibliography{expodiff}

\bibliographystyle{alpha}

\end{document}